\documentclass[twocolumn]{autart}    

\usepackage{algorithm, algorithmic}

\let\oldtheoremstyle\theoremstyle 
\let\theoremstyle\relax  
\usepackage{amsmath, amssymb, amsthm}
\let\theoremstyle\oldtheoremstyle
\usepackage{natbib}
\usepackage{array}
\usepackage{color}
\usepackage{graphicx}          

\newcommand{\ub}{\mathbf{u}}
\newcommand{\wb}{\mathbf{w}}
\newcommand{\xb}{\mathbf{x}}

\newcommand{\Ab}{\mathbf{A}}
\newcommand{\Bb}{\mathbf{B}}
\newcommand{\Cb}{\mathbf{C}}
\newcommand{\Db}{\mathbf{D}}

\newcommand{\Gb}{\mathbf{G}}
\newcommand{\Hb}{\mathbf{H}}
\newcommand{\Ib}{\mathbf{I}}

\newcommand{\Kb}{\mathbf{K}}
\newcommand{\Lb}{\mathbf{L}}

\newcommand{\Nb}{\mathbf{N}}
\newcommand{\Pb}{\mathbf{P}}
\newcommand{\Qb}{\mathbf{Q}}
\newcommand{\Rb}{\mathbf{R}}

\newcommand{\Wb}{\mathbf{W}}
\newcommand{\Xb}{\mathbf{X}}

\newcommand{\Zb}{\mathbf{Z}}

\newcommand{\Pc}{\mathcal{P}}

\newcommand{\Sc}{\mathcal{S}}

\newcommand{\norm}[1]{\left\lVert#1\right\rVert}

\newtheorem{theorem}{Theorem}

\newtheorem{assumption}{Assumption}

\newtheorem{definition}{Definition}

\newtheorem{lemma}[theorem]{Lemma}

\newtheorem{remark}{Remark}

\begin{document}

\begin{frontmatter}

\title{Regret Analysis of Policy Optimization over Submanifolds\\ for Linearly Constrained Online LQG\thanksref{footnoteinfo}} 

\thanks[footnoteinfo]{T.J. Chang and S. Shahrampour are with the Department of Mechanical and Industrial Engineering, Northeastern University, Boston, MA 02115, USA.}

\author[Boston]{Ting-Jui Chang}\ead{chang.tin@northeastern.edu} and     
\author[Boston]{Shahin Shahrampour}\ead{s.shahrampour@northeastern.edu}               

\address[Boston]{Department of Mechanical and Industrial Engineering, Northeastern University, Boston, USA}  


\begin{abstract}                          
Recent advancement in online optimization and control has provided novel tools to study online linear quadratic regulator (LQR) problems, where cost matrices are time-varying and unknown in advance. In this work, we study the online linear quadratic Gaussian (LQG) problem over the manifold of stabilizing controllers that are linearly constrained to impose physical conditions such as sparsity. By adopting a Riemannian perspective, we propose the online Newton on manifold (ONM) algorithm, which generates an online controller on-the-fly based on the second-order information of the cost function sequence. To quantify the algorithm performance, we use the notion of regret, defined as the sub-optimality of the algorithm cumulative cost against a (locally) minimizing controller sequence. We establish a regret bound in terms of the path-length of the benchmark minimizer sequence, and we further verify the effectiveness of ONM via simulations.
\end{abstract}

\end{frontmatter}

\section{Introduction}
LQR is one of the most well-studied optimal control problems in control theory \citep{4309169} with various application domains, such as econometrics, robotics, and physics. In LQR it is well-known that if the system dynamics and cost matrices are known, the optimal controller can be derived by solving Riccati equations. However, this assumption does not hold in {\it non-stationary} environments, where the cost parameters may change over time and are {\it unknown in advance}. In contrast to offline LQR, in which the objective is to compute an optimal controller based on time-invariant cost matrices, the typical goal in {\it online} LQR is to design a controller on-the-fly that can effectively adapt to the characteristics of the {\it time-varying} cost sequence. 

To study online control in general, a commonly used criterion for measuring performance is {\it regret}, defined as the accumulated sub-optimality (excessive cost) over time with respect to a benchmark policy. For online LQR problems, various approaches are proposed to reformulate the online control problem as an online optimization and use that framework to generate a {\it real-time} controller (e.g., semi-definite programming (SDP) relaxation for time-varying LQR \citep{cohen2018online} and noise feedback policy design for time-varying convex costs \citep{agarwal2019online}). 

Despite offering favorable theoretical guarantees in the form of sublinear regret with respect to the time horizon  $T$, existing work on online control \citep{cohen2018online,simchowitz2020naive,agarwal2019online,agarwal2019logarithmic,simchowitz2020improper,chang2021distributed} typically addresses {\it unconstrained} online controllers, parameterized as linear functions of system states or past noises. Unfortunately, these settings do not apply to {\it constrained} control problems where sparsity requirements are imposed on the controller matrix to reflect a physical condition or to capture the underlying interaction topology among various subsystems (e.g., coordination of unmanned aerial vehicles \citep{sarsilmaz2021distributed} and manipulation of robots in smart factories \citep{duan2024framework}).

In this work, we focus on the online LQG problem over the manifold of stabilizing controllers that are {\it linearly constrained}. The consideration of additional constraints makes the analysis more challenging as the domain of a constrained LQR is generally disconnected \citep{feng2019exponential}, preventing the {\it gradient dominance} property \citep{fazel2018global} that guarantees global convergence for unconstrained LQR. For example, in the context of unconstrained offline LQR, first-order methods have been shown to converge to the global optimal controller based on the gradient dominance property. But in the constrained setup, projected gradient descent techniques can only converge sublinearly to first-order {\it stationary} points \citep{bu2019lqr}. Note also that while some structures can be imposed on the controller through {\it regularization}, this class of methods only {\it promotes} the structural constraints rather than enforcing hard constraints. 

To address the constrained online LQG problem, our work takes a Riemannian perspective inspired by the recent work of  \citet{talebi2023policy}, where a second-order method was proposed based on the Riemannian metric arising from the optimal control problem itself. They showed that the Newton method based on this problem-related Riemannian metric can effectively capture the geometry of the cost in offline LQR, allowing the iterates to converge linearly (and eventually quadratically) to a local minimum. This favorable convergence behavior is partially attributable to the fact that the Riemannian hessian, defined with respect to the Riemannian metric, remains positive-definite on a larger domain compared to the Euclidean hessian.

In this work, we extend this idea to the online setup to develop a real-time controller which satisfies some linear structural constraints. The contributions of this work are as follows:
\begin{itemize}
    \item Inspired by \cite{talebi2023policy}, for linearly constrained online LQG problems, we propose the online Newton on manifold (ONM) algorithm, which is an online Riemannian metric-based second-order approach that leverages the inherent problem geometry while taking into account linear constraints imposed on the controller.
    \item Instead of comparing to a fixed control policy in hindsight, which is typical in online control problems, we consider a dynamic benchmark for regret. In particular, we define regret with respect to a sequence of (locally) minimizing linear policies. We then establish a dynamic regret bound based on the path-length of this benchmark minimizer sequence (Theorem \ref{T: Main theorem}). 
    \item We present several simulations to showcase the performance of our proposed algorithm: 1) For the constrained case, we illustrate that ONM is superior to its Euclidean-metric counterpart, as well as the projected gradient method. 2) For the unconstrained case, we compare ONM with two existing online control algorithms for LQR \citep{cohen2018online,agarwal2019online}. 3) We also validate our theory by demonstrating that increased fluctuation in the cost sequence results in greater regret.
\end{itemize}

\section{Related Literature}
\noindent
{\bf I) Linear Quadratic Regulator:}

{\bf I-A) Known Systems:} Based on the availability of the system states, the policy optimization can be classified as state-feedback LQR (SLQR) or output-feedback LQR (OLQR). For OLQR, the linear policy in terms of output was first studied in \citep{levine1970determination}, where the authors derived the necessary condition for the optimal controller and provided an iterative policy learning method that solves a series of nonlinear matrix equations at each iteration. Following this work, continued effort was made to address this problem through the lens of first- and second-order methods, and convergence guarantees were also provided with the help of backtracking techniques \citep{moerder1985convergence, makila1987computational, toivonen1987newton, rautert1997computational}. For SLQR, the convergence property of gradient-based methods has been studied for both discrete-time and continuous-time cases \citep{fazel2018global,bu2019lqr,bu2020policy,mohammadi2021convergence}, where the LQR costs were shown to satisfy the gradient dominance property, which ensures linear convergence rates despite nonconvexity. 

To model the network inter-connection in distributed control problems, structural constraints are sometimes encoded as linear constraints imposed on the control gain; however, for general constrained LQR problems, the gradient dominance property does not necessarily hold, and gradient-based methods such as Projected Gradient Descent (PGD)  converge to first-order stationary points with a sublinear rate \citep{bu2019lqr}. Recently, \citet{talebi2023policy} proposed a second-order update method for linearly constrained LQR, where the Hessian operator is defined based on the Riemannian metric arising from the optimization objective, and they proved that their method converges locally linearly (and eventually quadratically).

{\bf I-B) Unknown Systems: }
For linear systems with unknown system parameters, various data-driven approaches were proposed for the learning part of LQR. Depending on whether or not the system identification procedure is incorporated into the learning process, these methods can be classified as {\it indirect} or {\it direct} (by bypassing the identification step). For indirect methods, \citet{dean2020sample} proposed an algorithm for which the cost sub-optimality gap grows linearly with the parameter estimation error. This dependence was later improved by \citet{mania2019certainty}, who presented a sub-optimality gap scaling quadratically with the estimation error.
For direct approaches, \citet{de2019formulas} proposed a new parameterization, which instead of using system matrices, formulates the problem in terms of the observation of state and input sequences. For the noisy setup, \citet{de2021low} demonstrated sufficient conditions under which a small relative error is guaranteed with respect to the unknown optimal controller, and in \citep{dorfler2023certainty}, a regularized method was proposed to promote certainty-equivalence. The data-based formulation for learning of Kalman gain was proposed in \citet{liu2024learning}. \citet{zhao2025data} proposed another parameterization, where the dimension of the policy depends only on the system dimension, and showed that the corresponding policy optimization enjoys the property of projected gradient dominance.

{\bf II) Online Optimization:} 

There exists a rich body of literature on the field of online optimization. The general goal of this problem is to make online decisions in the presence of a time-varying sequence of functions that may change adversarially. As mentioned earlier, the performance of an online algorithm is captured by the notion of {\it regret}, which is considered to be static (dynamic) if the benchmark decision is fixed (time-varying). For the static case, it is well-known that the optimal regret bounds are $O(\sqrt{T})$ and $O\left(\log(T)\right)$ for convex and strongly convex functions, respectively \citep{zinkevich2003online,hazan2007logarithmic}. For the dynamic case, as the function sequence may vary in an arbitrary manner, typically there are no explicitly sublinear regret bounds. Instead, the dynamic regret bound is presented in terms of different regularity measures of the benchmark sequence: path-length \citep{zinkevich2003online,zhang2018adaptive,mokhtari2016online}, function value variation \citep{besbes2015non}, and variation in gradients or Hessians \citep{chiang2012online,rakhlin2013optimization,jadbabaie2015online,chang2021online}.

{\bf III) Online Control:} 

Recent advancements in online optimization and control have fueled interest in studying linear dynamical systems with time-varying costs. For linear time-invariant (LTI) systems with known dynamics, \citet{cohen2018online} reformulated the online LQG problem with SDP relaxation and established a regret bound of $O(\sqrt{T})$. The setup was later extended to the case with general convex functions and adversarial noises in \citet{agarwal2019online}, where the disturbance-action controller (DAC) parameterization was proposed, and a regret bound of $O(\sqrt{T})$ was derived. The regret bound was later improved to $O\left(\text{poly}(\log(T))\right)$ in \cite{agarwal2019logarithmic} when strongly convex functions were considered. In addition to the case of known linear dynamics, the setup of {\it unknown} dynamics was also studied for convex costs \citep{hazan2020nonstochastic} and strongly convex costs \citep{simchowitz2020improper}. \citet{zhao2022non} studied the setup with general convex costs for LTI systems and derived a dynamic regret bound with respect to a time-varying DAC. This bound was later improved for the case of quadratic costs \citep{baby2022optimal}. For linear time-varying systems, the corresponding dynamic regret bound was derived by \citet{luo2022dynamic}. Finally, another setup, where constraints are imposed on states and control actions, was  studied  by \citep{li2021online,li2021information} to model the safety concerns. 
\section{Problem Formulation}
In this section, we provide the problem formulation as well as background information on tools we use to design and analyze ONM. 
\begin{itemize}
    \item In Section \ref{subsec: problem formulation}, we formulate the linearly constrained online LQG control problem and define the performance criterion to assess our proposed algorithm (ONM).
    \item In Section \ref{subsec: idea of strong stability}, to characterize the performance of the online linear controllers, including ONM, we present the idea of (sequential) strong stability based on \cite{cohen2018online}.
    \item To better leverage the intrinsic geometry of LQR, the problem is transformed into an online Riemannian optimization. In Section \ref{subsec: Riemannian background}, we first introduce a Riemannian metric arising from LQR. Then, to define gradient and Hessian on the Riemannian submanifold, we discuss the idea of Riemannian connection \citep{talebi2023policy}. 
\end{itemize}

\subsection{Notation} 
{\small
\begin{tabular}{|c||m{20em}|}
    \hline
    $\rho(\Ab)$ & The spectral radius of matrix $\Ab$ \\
    \hline
    $\norm{\cdot}$ & Euclidean (spectral) norm of a vector (matrix)\\
    \hline
    $\norm{\cdot}_g$ & Norm induced by the Riemannian metric $g$\\
    \hline
    $\text{dist}(\cdot,\cdot)$ & Riemannian distance based on metric $g$\\
    \hline
    $\mathbb{E}[\cdot]$ & The expectation operator\\
    \hline
    $\mathbb{L}(\Ab,\Zb)$ & Solution of the Lyapunov equation: $\Xb = \Ab\Xb\Ab^{\top} + \Zb$\\
    \hline
    $\Ib_d$ & Identity matrix with dimension $d\times d$\\
    \hline
    $\underline{\lambda}(\Ab)$ & The minimum eigenvalue of $\Ab$\\
    \hline
    $\overline{\lambda}(\Ab)$ & The maximum eigenvalue of $\Ab$\\
    \hline
     $T_{\Kb}\Sc$ & The tangent space at point $\Kb$ on manifold $\Sc$\\
    \hline    
\end{tabular}}

Throughout the paper, when the inner product is used, the corresponding metric is clear from the context.

\subsection{Linearly Constrained Online LQG Control}\label{subsec: problem formulation}
We consider an LTI system with the following dynamics
\begin{equation*}
    \xb_{t+1} = \Ab\xb_t + \Bb\ub_t + \wb_t,
\end{equation*}
where the system matrices $\Ab \in \mathrm{R}^{n \times n}$ and $\Bb \in \mathrm{R}^{n \times m}$ are known and $\wb_t$ is a Gaussian noise with zero mean and covariance $\Wb \succeq \sigma^2 \Ib$. The noise sequence $\{\wb_t\}$ is assumed to be independent and identically distributed over time. The general goal of an online LQG problem is that given a sequence of cost matrices $\{(\Qb_t,\Rb_t)\}$, which is {\it unknown} in advance and is revealed sequentially to the algorithm, decide the control signal in {\it real time} while ensuring an acceptable cumulative quadratic cost. In other words, in round $t$, an online algorithm receives the state $\xb_{t}$ and applies the control $\ub_{t}$. Then, the positive-definite cost matrices $\Qb_{t}$ and $\Rb_{t}$ are revealed, and the cost $\xb_{t}^\top\Qb_{t}\xb_{t} + \ub_{t}^\top\Rb_{t}\ub_{t}$ is incurred. Throughout this paper, we assume that there exists some $C>0$, such that $\text{Tr}(\Qb_t), \text{Tr}(
\Rb_t)\leq C$. Note that if the sequence of cost matrices is known in advance, the optimal controller can be easily derived by solving the Riccati equation.

{\bf Linearly Constrained Controllers:} 
Given a stabilizable linear dynamical system $(\Ab,\Bb)$, we define the set of stable linear controllers as follows
\begin{equation*}
    \Sc :=\{\Kb\in \mathrm{R}^{m\times n}\;|\;\rho(\Ab+\Bb\Kb)<1\}. 
\end{equation*}
In the existing literature on optimal control of LTI systems, the optimal controller is generally achieved by searching $\Sc$, which mainly leads to a dense solution that may violate some practical conditions, e.g., the sparsity requirement or safety restrictions imposed by the physical constraints. Following \cite{talebi2023policy}, in this work we take into account such constraints by considering some linear constraints on $\Kb$, e.g., $\Cb\Kb = \Db$, and we seek to learn a sequence of linear controllers $\{\Kb_t\}$ $(\ub_t = \Kb_t\xb_t)$, such that $\Kb_t\in \Tilde{\Sc}:=\Sc\cap \mathcal{K}$, where $\mathcal{K}:=\{\Kb\in \mathrm{R}^{m\times n}\;|\;\Cb\Kb=\Db\}$.\\\\
{\bf Regret Definition:} For any online LQG control algorithm $\mathcal{A}$, the corresponding cumulative cost after $T$ steps is expressed as 
\begin{equation}\label{Eq: Cumulative cost of an online control alg}
    J_T(\mathcal{A})=\mathrm{E}\left[\sum_{t=1}^T {\xb_t^{\mathcal{A}}}^\top\Qb_t\xb_t^{\mathcal{A}} + {\ub_t^{\mathcal{A}}}^\top\Rb_t\ub_t^{\mathcal{A}}\right].
\end{equation}
Since the costs are {\it unknown} in advance and the controllers are determined in real time, it is not possible to directly minimize the cumulative cost and quantify the exact sub-optimality. To gauge the performance of $\mathcal{A}$, we use the notion of {\it regret}, defined as the difference between the cumulative cost and the cost associated with a comparator policy $\pi$ as
\begin{equation}\label{Eq: Regret}
    \text{Regret}_T(\mathcal{A}):= J_T(\mathcal{A}) - J_T(\pi).
\end{equation}
In this work, the comparator policy $\pi$ is defined as the sequence of linear controllers $\{\Kb^*_t\}$, such that $\forall t$, $\Kb^*_t$ is a local minimizer (over $\Tilde{\Sc}$) of the   {\it time-invariant infinite-horizon} LQG problem with $(\Qb_t,\Rb_t)$ as the corresponding cost matrices. Note that this comparator policy is greedy in the sense that if the cost matrices of the original time-varying problem stay constant after step $t_0$, then the policy $\ub_{t} = \Kb^*_{t_0}\xb_{t}$ enjoys a (locally) optimal performance when $T$ goes to infinity. Regret is a standard metric in online decision making, which has also been used as an indicator of the system {\it stability} in control \citep{karapetyan2023implications}.

\subsection{Strong Stability and Sequential Strong Stability}\label{subsec: idea of strong stability}
In order to capture the performance of online optimal control algorithms, following \cite{cohen2018online}, we introduce the notion of strong stability.
\begin{definition}\label{D: Strong Stability}(Strong Stability) A linear policy $\Kb$ is $(\kappa, \gamma)$-strongly stable (for $\kappa > 0$ and $0<\gamma\leq 1$) for the LTI system $(\Ab,\Bb)$, if $\norm{\Kb}\leq \kappa$, and there exist matrices $\Lb$ and $\Hb$ such that $\Ab+\Bb\Kb=\Hb\Lb\Hb^{-1}$, with $\norm{\Lb}\leq 1-\gamma$ and $\norm{\Hb}\|\Hb^{-1}\|\leq \kappa$.
\end{definition}

The idea of strong stability simply provides a quantitative perspective of stability. In fact, any stable controller can be shown to be strongly stable for some $\kappa$ and $\gamma$ \citep{cohen2018online}. The notion of strong stability helps with quantifying the rate of convergence to the steady-state distribution. In addition to strong stability, as the applied controller $\Kb_t$ changes over time in the online setup, we also use {\it sequential} strong stability, defined in \citet{cohen2018online} as follows.
\begin{definition}\label{D: Sequential Strong Stability}(Sequential Strong Stability) A sequence of linear  policies $\{\Kb_t\}_{t=1}^T$ is $(\kappa,\gamma)$-strongly stable if there exist matrices $\{\Hb_t\}_{t=1}^T$ and $\{\Lb_t\}_{t=1}^T$ such that $\Ab+\Bb\Kb_t=\Hb_t\Lb_t\Hb_t^{-1}$ for all $t$ with the following properties,
\begin{enumerate}
    \item $\norm{\Lb_t}\leq 1-\gamma$ and $\norm{\Kb_t}\leq \kappa$.
    \item  $\norm{\Hb_t}\leq \beta^{\prime}$ and $\norm{\Hb_t^{-1}}\leq 1/\alpha^{\prime}$ with $\kappa=\beta^{\prime}/\alpha^{\prime}$.
    \item $\norm{\Hb_{t+1}^{-1}\Hb_t}\leq 1+\gamma/2$.
\end{enumerate}
\end{definition}
With the idea of sequential strong stability, for the time-varying control policy $\ub_t = \Kb_t\xb_t$, we can quantify the difference between the sequence of state covariance matrices induced by the policy and the sequence of steady-state covariance matrices. 

\subsection{Policy Optimization over Manifolds for Time-Invariant LQ Control}\label{subsec: Riemannian background}

In \citep{talebi2023policy}, it was shown that if a time-invariant LQG (or LQR) problem is formulated as an optimization over a manifold, equipped with the Riemannian metric arising from the original problem, then the update direction computed using the second-order information defined on this problem-oriented Riemannian metric can better capture the inherent geometry, which in turn provides a better convergence to a local minimum. In this section, we highlight the key idea of the Riemannian approach proposed in \citep{talebi2023policy}.

{\bf Riemannian Metric:} Let us consider $\Sc$ (the set of stable linear policies) as a manifold on its own. Consider the cost function of a time-invariant infinite-horizon LQG control problem with a fixed cost pair $(\Qb,\Rb)$, which for a linear policy $\Kb$ can be expressed as $f(\Kb) = \text{Tr}(\Pb_{\Kb}\Wb)$, where  
\begin{equation}\label{Eq: Cost reformulation}
\begin{split}
    \Pb_{\Kb}&:=\mathbb{L}\big((\Ab+\Bb\Kb)^{\top},\Qb + \Kb^{\top}\Rb\Kb\big).
\end{split}
\end{equation}
Then, based on the Lyapunov-trace property, the cost can be reformulated as 
\begin{equation}\label{Eq: Cost reformulation2}
  f(\Kb)=\text{Tr}\left((\Qb + \Kb^{\top}\Rb\Kb)\mathbb{L}(\Ab+\Bb\Kb, \Wb)\right).  
\end{equation}
Inspired by this expression, \cite{talebi2023policy} proposed a  covariant 2-tensor field, which was shown to be a Riemannian metric that better captures the geometry of LQG problems (see Proposition 3.3 in \cite{talebi2023policy}). This Riemannian metric, which we denote by $g$, provides a larger region in which the Hessian is positive-definite, and it can be useful for the convergence of second-order algorithms. Throughout this paper, for any two linear controllers $\Kb_1, \Kb_2 \in \Sc$, we denote their Riemannian distance (with respect to $g$) as $\text{dist}(\Kb_1,\Kb_2)$.

{\bf Riemannian Gradient and Hessian:} We denote by $\text{grad} f$ the gradient of $f$ with respect to the metric $g$. Then, the Hessian operator of any smooth function $f\in C^{\infty}(\Sc)$, is defined as
\begin{equation*}
    \text{Hess }f[U] := \nabla_U \text{grad } f, 
\end{equation*}
where $\nabla_U$ denotes the covariant derivative operator along the vector field $U$. As mentioned earlier, in practice we are more interested in controllers $\Kb$ that are in a relatively simple subset $\mathcal{K}\subset R^{n\times m}$ such that $\Tilde{\Sc}:=\mathcal{K}\cap \Sc$ is an embedded submanifold of $\Sc$, so we focus on the restriction of $f$ to $\Tilde{\Sc}$ denoted by $h$, i.e., $h:=f|_{\Tilde{\Sc}}$.  \cite{talebi2023policy} showed that based on the Riemannian tangential and normal projections, we can calculate the Riemannian gradient $\text{grad } h$ as well as the Hessian operator $\text{Hess }h[U]$ once we know the same for $f$  (see Proposition 3.5 in \citep{talebi2023policy}).

{\bf Optimization over Riemannian Manifolds: }If we want to directly solve the linearly constrained LQG using existing techniques developed for optimization over manifolds, it is necessary to use a retraction operation. However, such a retraction is generally not available due to the complex geometry of $\Sc$. Although it is possible to derive the Riemannian exponential map and use it as a retraction, the computation involves solving a system of second-order ordinary differential equations, i.e., geodesic equations based on the Christoffel symbols of the Riemannian metric $g$, which is computationally undesirable (see Proposition 3.4 in \citep{talebi2023policy} for explicit expressions). To address this issue in their proposed algorithm, \citet{talebi2023policy} control the update direction and enforce the feasibility of the iterate updated along this direction by choosing a proper step-size with a stability certificate, defined as follows.
\begin{lemma}\label{L: Stability certificate}(Lemma 4.1 in \cite{talebi2023policy})
    Consider a smooth mapping $\mathcal{Q} : \Sc \to \mathrm{R}^{n\times n}$ that sends $\Kb$ to any $\mathcal{Q}_{\Kb} \succ 0$. For any direction $\Gb \in T_{\Kb}\Sc$ at any point $\Kb \in \Sc$, if
    \begin{equation*}
        0\leq \eta\leq s_{\Kb}:=\frac{\underline{\lambda}(\mathcal{Q}_{\Kb})}{2\overline{\lambda}\big(\mathbb{L}((\Ab+\Bb\Kb)^{\top},\mathcal{Q}_{\Kb})\big)\norm{\Bb\Gb}_2}, 
    \end{equation*}
    then $\Kb + \eta\Gb \in \Sc$. $s_{\Kb}$ is referred to as the stability certificate at $\Kb$.
\end{lemma}
The stability certificate provides a condition number that depends on geometric information of the manifold at point $\Kb\in \Sc$. This certificate indeed depends on the mapping $\mathcal{Q}$, which can be chosen arbitrarily as long as it is positive-definite. One such choice is $\mathcal{Q}_{\Kb}=\Qb+\Kb^\top \Rb \Kb$.

\begin{algorithm}[tb]
\caption{Online Newton on Manifold}
\label{alg:ONM}
\begin{algorithmic}[1]
    \STATE {\bfseries Require:} system parameters $(\Ab,\Bb)$, linear constraint $\mathcal{K}$, smooth mapping $\mathcal{Q}$, time horizon $T$.
   
    \STATE {\bf Initialize:} $\Kb_1$ close enough to $\Kb^*_1$ in terms of the Riemannian distance.
    
    \FOR{$t=1,2,\ldots,T$}
        \STATE Apply the control $\ub_t = \Kb_t\xb_t$ and receive $(\Qb_t,\Rb_t)$.

        \STATE Find the Newton direction ${\Gb}_{t}$ on $\Tilde{\Sc}$ satisfying
        $$\text{Hess } {h_t}_{\Kb_{t}} [\Gb_{t}] = -\text{grad }{h_t}_{\Kb_{t}},$$
 where $h_t:=f_t|_{\Tilde{\Sc}}$ and $f_t$ is defined in \eqref{Eq: Cost reformulation2} with matrices $\Qb_t$ and $\Rb_t$.
        \STATE Compute the stability certificate $s_{\Kb_{t}}$, choose step-size $\eta_{t}=\min\{1,s_{\Kb_{t}}\}$ and perform the update
        $$\Kb_{t+1} = \Kb_{t} + \eta_{t}\Gb_{t}.$$
    \ENDFOR
\end{algorithmic}
\end{algorithm}

\section{Algorithm: Online Newton on Manifold}
In this section, we present an algorithm for the online LQG problem, where the applied linear controller $\Kb_t$ needs to satisfy some linear constraints, e.g., $\Cb\Kb_t = \Db$. Our approach is to formulate the linearly constrained online LQG problem as an online optimization with the function sequence $\{h_t = f_t|_{\Tilde{\Sc}}\}$, where $f_t(\Kb)$ denotes the time-averaged infinite-horizon LQG cost based on $(\Qb_t, \Rb_t)$, following Equation \eqref{Eq: Cost reformulation2}.  The proposed algorithm, which we call online Newton on manifold (ONM), is summarized in Algorithm \ref{alg:ONM}. The core idea of ONM is to leverage second-order information derived with respect to the problem-oriented Riemannian metric to better capture the non-Euclidean geometry.
Each ONM iteration consists of two parts: 1) After receiving the  function information $(\Qb_t,\Rb_t )$ for round $t$, the algorithm computes the update direction $\Gb_t$ based on the Hessian operator defined by the Riemannian metric. 2) To ensure the feasibility of the updated controller, the step-size $\eta_t$ is derived based on the stability certificate. Then, the current controller is updated by $\eta_t\Gb_t$ and applied in the next iteration. 

The computational cost of ONM greatly depends on the form of linear constraint set $\mathcal{K}$. For example, suppose that we consider the sparsity constraint, such that the controller has $|D|$ non-zero elements, where $0\leq |D| \leq nm$. Then, the computation cost of each iteration can be decomposed into the following parts: 1) As the Riemannian metric is location-varying, for each iterate $\Kb_t$, the metric tensor needs to be computed based on Proposition 3.3 in \citep{talebi2023policy}), and the corresponding cost is $O(n^3)$ (solving the Lyapunov equation). 2) The Newton direction on the submanifold is computed using the Hessian and gradient operators defined by the Riemannian tangential projection, and the resulting cost is $O(n^3|D| + |D|^3)$ (detailed expressions are provided in Section V of \citep{talebi2023policy}). 3) The calculation of the stability certificate at $\Kb_t$ takes $O(n^3)$ operations. Therefore, the total computation cost for each iteration is $O(n^3|D| + |D|^3)$, which is at most $O(n^4m+n^3m^3)$ when $|D|=nm$.

\section{Theoretical Results}
With the help of the problem-oriented Riemannian metric, we show that ONM can effectively adapt to the dynamic environment with a regret guarantee in terms of the path-length of the (locally) optimal controller sequence $\{\Kb^*_t\}$, and the regret is sublinear when the sequence is slowly varying. 

Let us start with stating our technical assumptions that  are quite standard for analyzing the local convergence of the second order methods.
\begin{assumption}\label{A: Locally optimal stationary}
    The local minimizer $\Kb^*_t$ is a nondegenerate local minimum of $h_t:=f_t|_{\Tilde{\Sc}}$ for all $t$.
\end{assumption}

\begin{assumption}\label{A: Hessian boundedness}
    For all $t$, there exists a compact neighborhood $\mathcal{U}_t \subset \Tilde{\Sc}$, where $\text{Hess }h_t$ is positive-definite. Also, there exist positive constants $\mu_g,L_g$ such that for all $\Kb\in \mathcal{U}_t$ and $\Gb\in T_{\Kb}\Tilde{\Sc}$,
    $$\mu_g\norm{\Gb}^2_{g_{\Kb}} \leq \langle \text{Hess }{h_t}_{\Kb} [\Gb],\Gb\rangle \leq L_g\norm{\Gb}^2_{g_{\Kb}}.$$ 
    We further assume that the Hessian is $L_H$-Lipschitz smooth. 
\end{assumption}

\begin{assumption}\label{A: Bounded trace of a stable controller}
    For all $\Kb\in \cup \{\mathcal{U}_t\}$, there exists a positive constant $\nu$ such that the corresponding steady-state covariance $\Xb$ satisfies $\text{Tr}(\Xb + \Kb\Xb\Kb^{\top}) \leq \nu$.
\end{assumption}
All these assumptions are mild in the sense that Assumptions \ref{A: Locally optimal stationary}-\ref{A: Hessian boundedness} are standard in the literature for the convergence analysis of Newton-type methods \citep{nesterov1998introductory, chang2021online, talebi2023policy}, and we just naturally use their Riemannian versions in our context. Moreover, Assumption \ref{A: Bounded trace of a stable controller} is only used for quantification purposes and can be easily justified based on the boundedness of $\cup \{\mathcal{U}_t\}$.

We now present our main theorem, which provides a regret guarantee for ONM.  

\begin{theorem}\label{T: Main theorem}
Suppose that Assumptions \ref{A: Locally optimal stationary} to \ref{A: Bounded trace of a stable controller} hold. Assume further that
\begin{enumerate}
    \item $\exists \Kb_1 \in \mathcal{U}_1$ such that $\text{dist}(\Kb_1,\Kb^*_1)\leq \vartheta$, where $\vartheta$ is defined as follows 
    \begingroup
    \small
    \begin{equation*}
        \vartheta:=\min \left\{\frac{\sigma^4}{L_s\nu(1+\alpha)},\frac{c}{(L_g/\mu_g)-\alpha/2},\frac{\alpha}{(L_HL_g^2/\mu_g^3)}\right\},
    \end{equation*}
    \endgroup
   where $\alpha \in (0,1/\sqrt{2})$, and $L_s, L_H, L_g \text{ and } c$ are problem-related constants (see Section \ref{sec:constants} for definitions).
    \item For all $t$, we have 
    \begin{equation*}
        \text{dist}(\Kb^*_{t+1}, \Kb^*_{t})\leq \min \left\{ \frac{\sigma^4}{L_s\nu}, (1-\alpha)\vartheta\right\}.
    \end{equation*}
\end{enumerate}    
Then, by choosing $\eta_t=\min\{1, s_{\Kb_{t}}\}$, where $s_{\Kb_{t}}$ is the stability certificate at $\Kb_t$, the regret of ONM is of the following order,
\begin{equation*}
    \text{Regret}_T(\text{ONM})=O\Big(\sum_{t=2}^{T}\text{dist}(\Kb^*_{t}, \Kb^*_{t-1})\Big).
\end{equation*}

\end{theorem}


Compared to previous work on online control \citep{cohen2018online,agarwal2019online,agarwal2019logarithmic,simchowitz2020improper,chang2021distributed}, where the (static) regret bounds are sublinear in terms of $T$, in our work, since the benchmark policy $\{\Kb^*_t\}$ is {\it time-varying}, we end up with a dynamic regret bound in terms of the path-length of this optimal controller sequence. With this perspective, we conclude that if the cumulative fluctuations satisfies $\sum_{t=2}^{T}\text{dist}(\Kb^*_{t}, \Kb^*_{t-1})=o(T)$, the cost grows sublinearly over time.  

\begin{remark}
    Note that the assumption on the initialization distance $\vartheta$ in Theorem \ref{T: Main theorem} is easily achievable by running backtracking line-search algorithms, as they have global convergence guarantees. 
\end{remark}
    We note that line 6 of the ONM algorithm, which uses stability certificate, can be replaced with backtracking line-search techniques while achieving similar regret guarantees, but this type of approach may introduce undesirable computational burden, especially in the online setup. Since ONM chooses the step-size based on the stability certificate, it provides more flexibility in coping with dynamic environments and allows a wider range of variation for $\text{dist}(\Kb^*_{t+1}, \Kb^*_{t})$, which is an advantage.
    


\section{Numerical Experiments}
In this section, we present our numerical experiments to demonstrate the effectiveness of ONM in practice.

\begin{figure}[t!] 
    \includegraphics[width=.95\columnwidth]{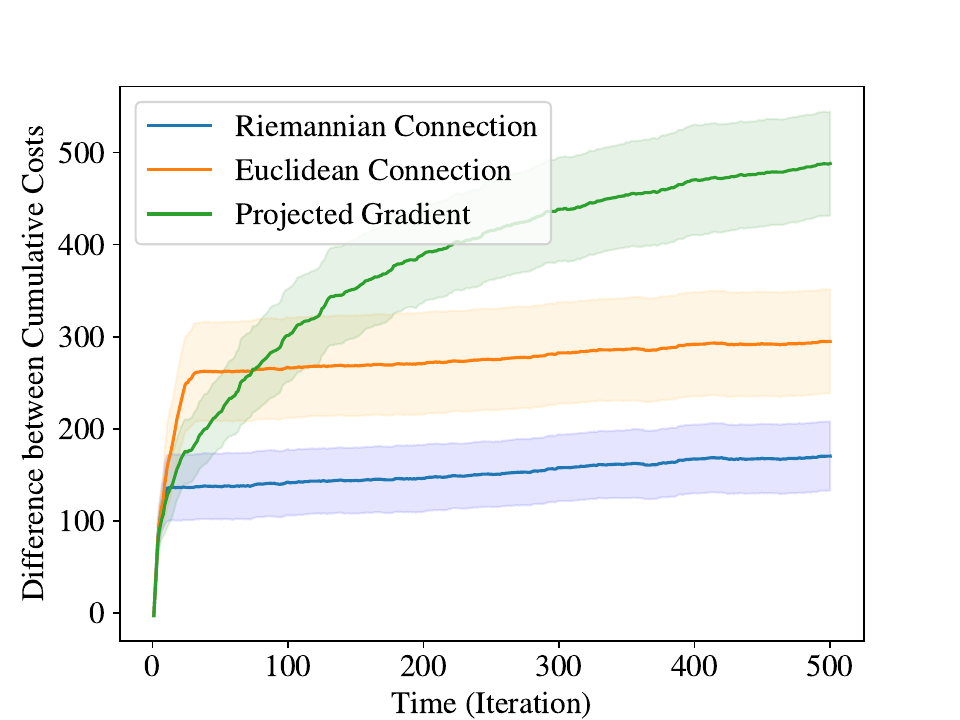}
    \caption{Regret for the constrained case.}
    \label{fig:Cumulative cost difference (constrained_case)}
\end{figure}

\begin{figure}[t!] 
    \includegraphics[width=.95\columnwidth]{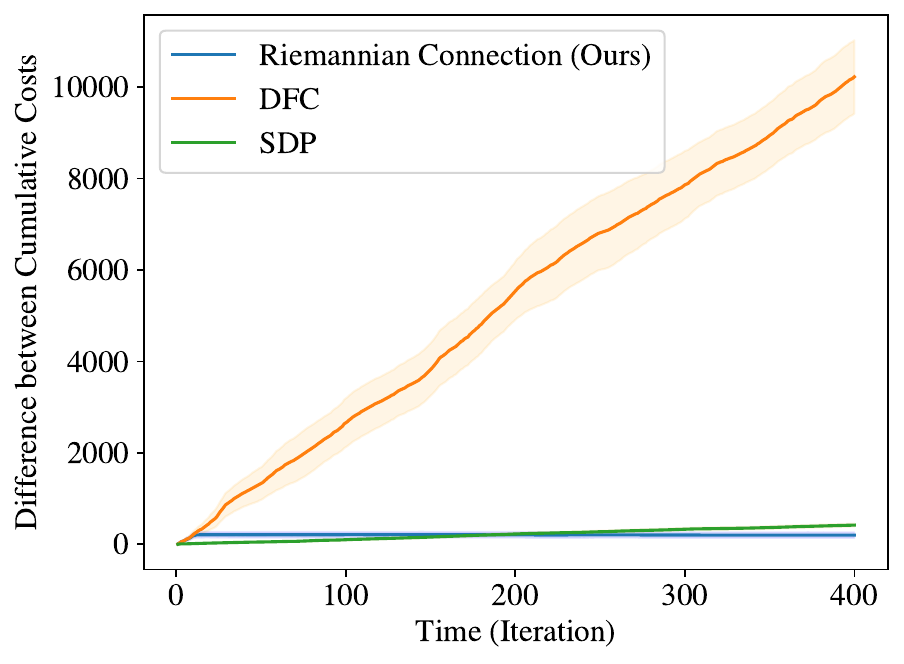}
    \caption{Regret for the unconstrained case.}
    \label{fig:Cumulative cost difference (unconstrained_case)}
\end{figure}

\begin{figure}[t!] 
    \includegraphics[width=.95\columnwidth]{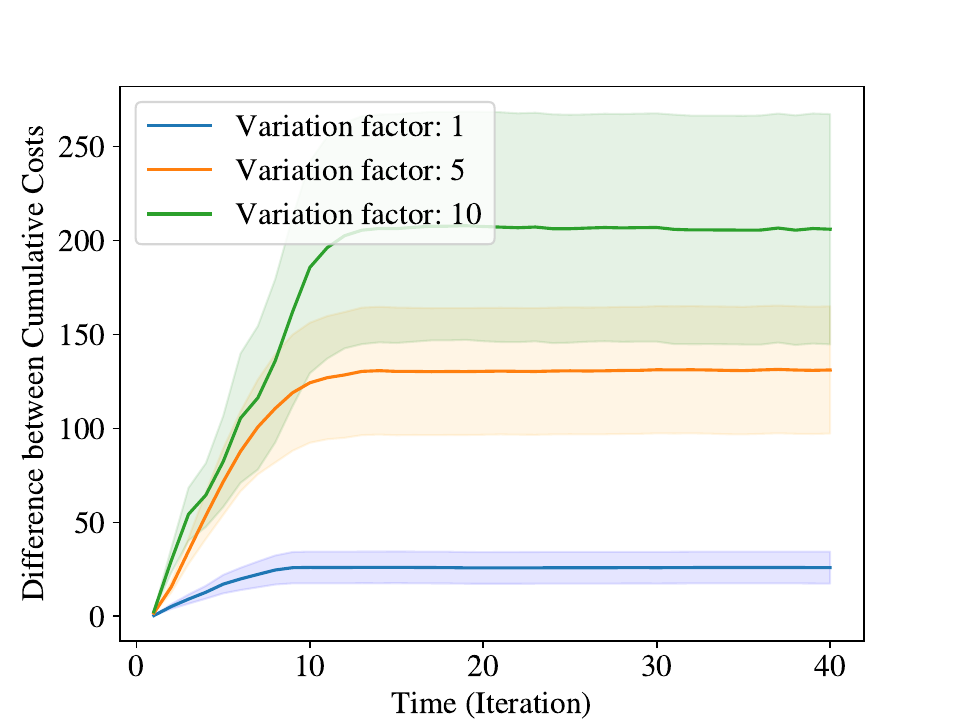}
    \caption{Regret for different variation factors.}
    \label{fig:Cumulative cost difference (function variations)}
\end{figure}

\noindent
We consider a dynamical system with state dimension $n=6$ and input size $m=3$. The elements of system matrices $(\Ab,\Bb)$ are sampled from a Normal distribution with zero mean and unit variance, and $\Ab$ is scaled to ensure that the system is open-loop stable. For the constraint $\mathcal{K}$, we consider the sparsity requirement that half of the elements of the controller matrix (randomly selected) are forced to be zero. For the function sequence $\{(\Qb_t,\Rb_t)\}$, to ensure that the corresponding (local) minimizer sequence varies slowly, $\Qb_t$ (and similarly $\Rb_t$) is constructed using the following formula
\begin{equation}\label{eq:cost}
    \Qb_t = \Ib + \text{variation factor}*\Nb_t^{\top}\Nb_t,
\end{equation}
where $\Nb_t$ is a diagonal noise matrix with elements generated from the uniform distribution on $(0,1)$, and the variation factor is a user-defined constant. We consider three scenarios: 

I) In the first experiment, we compare three different online approaches: 1) ONM; 2) a second-order method where the Hessian operator is defined based on the Euclidean connection; 3) the projected gradient method (PG). As there is no closed-form solution for the constrained setup, we compute the minimizer sequence $\{\Kb^*_t\}$ numerically by running the method in \citep{talebi2023policy} until the gradient norm is smaller than a given threshold. Also, the step-size for these three applied approaches is chosen to satisfy the stability certificate. Given the predefined $(\Ab,\Bb)$, $\{(\Qb_t,\Rb_t)\}$ and the sparsity requirement, we repeat 30 Monte-Carlo simulations and compute the expected regret, which is the difference between the cumulative cost of the corresponding algorithm and that of the algorithm using $\{\Kb^*_t\}$. From Fig. \ref{fig:Cumulative cost difference (constrained_case)} we can see that the regret for PG is worse than the second-order methods since the update is solely based on first-order information. Also, the superior performance of ONM over the Euclidean connection is expected as the Riemannian connection is compatible with the metric arising from the inherent geometry.

II) We also consider the unconstrained setup (without constraint $\mathcal{K}$), where we compare ONM with two other methods, namely the disturbance feedback policy (DFC) \citep{agarwal2019online} and the SDP relaxation approach \citep{cohen2018online}. The step-size for these two methods is chosen based on their corresponding theorems. For the unconstrained setup, the comparator sequence $\{\Kb^*_t\}$ is the exact optimal controller sequence derived by solving the Riccati equation. Again, we run 30 Monte-Carlo simulations to cover different realizations of the system noise. Although we cannot directly compare these three methods since each of them applies a different parameterization, we can see that ONM is capable of quickly adapting to the dynamic environment (Fig. \ref{fig:Cumulative cost difference (unconstrained_case)}). In the experiments, we also observe that since controller parameterizations of both ONM and DFC require a pre-given stable linear controller, the performance also depends on the stability of this given stable controller, which explains the larger regret of ONM compared to that of SDP in the early stage. 
 
 III) Lastly, we evaluate the performance of ONM under different levels of function variations by adjusting the variation factor in \eqref{eq:cost} (Fig. \ref{fig:Cumulative cost difference (function variations)}). We can see that as the variation of the function increases, the regret becomes worse, since the minimizer sequence has more fluctuations, which is in alignment with our theory.

\section{Conclusion}
In this work, we studied the linearly constrained online LQG problem and proposed the ONM algorithm, which is an online second-order method based on the problem-related Riemannian metric. To quantify the performance of ONM, we presented a dynamic regret bound in terms of the path-length of the minimizer sequence of a time-varying infinite-horizon LQG. We also provided simulation results showing the superiority of ONM compared to Newton method with Euclidean metric and projected gradient descent, as well as SDP relaxation and DFC for online control. For future directions, it is interesting to explore the decentralized extension of the problem to accommodate online control in multi-agent systems. Also, another possible direction is to investigate the {\it unknown} dynamics setup and study a Riemannian metric that is built on system estimates.


\appendix{\bf Appendix}

The Appendix consists of three sections. We present some of the important constants in our analysis (Section \ref{sec:constants}), the proof of our main theorem (Section \ref{Sec: A}), and the auxiliary lemmas useful for the proof (Section \ref{Sec: B}).

\section{Constant Terms}\label{sec:constants}
\begin{enumerate}
    \item For any $\Kb_1, \Kb_2 \in \cup \{\mathcal{U}_t\}$ and their corresponding steady-state covariance matrices $\Xb^s_1$ and $\Xb^s_2$,
    \begin{equation*}
        \norm{\Xb_1^s - \Xb_2^s} \leq L_s \text{dist}(\Kb_1, \Kb_2),
    \end{equation*}
    where $\text{dist}(\cdot,\cdot)$ is the Riemannian distance based on the metric $g$.
    \item $L_g$ and $\mu_g$ are defined in Assumption \ref{A: Hessian boundedness}.
    \item Given any $\Kb \in \cup \{\mathcal{U}_t\}$ and any $\Gb \in T_\Kb\Tilde{\Sc}$, define a curve $r: [0,s_{\Kb}]\to \Tilde{\Sc}$ such that $r(\tau) = \Kb + \tau \Gb$. Then, $\forall t$ we have 
     \begingroup
\scriptsize
    \begin{equation*}
        \begin{split}
        \norm{(\text{Hess }{h_t}_{r(\tau)}-\Pc^r_{0,\tau}\text{Hess }{h_t}_{r(0)})[\Gb]}_{g_{r(\tau)}}\leq &L_H\tau\norm{\Gb_{t}}^2_{g_{\Kb}},
    \end{split}   
    \end{equation*}
    \endgroup
    where $\Pc^r_{0,\tau}$ denotes the parallel transport operator from $0$ to $\tau$ along the curve $r$. 
    \item Constant $c$ is used to provide a lower bound for the stability certificate and it first appears in \eqref{Local convergence Eq14}.
\end{enumerate}

\section{Proof of Theorem \ref{T: Main theorem}}\label{Sec: A}
Let us start by introducing some notation. Consider 
\begin{equation*}
\begin{split}
\xb_{t+1} &= (\Ab+\Bb\Kb_t)\xb_t + \wb_t ~~~~~~~~\Xb_t:=\mathbb{E}\left[\xb_t\xb_t^{\top}\right]\\
\xb^*_{t+1} &= (\Ab+\Bb\Kb^*_t)\xb^*_t + \wb_t ~~~~~~~\Xb^*_t:=\mathbb{E}\left[\xb_t^*\xb_t^{*\top}\right]. 
\end{split}
\end{equation*}
Also, let $\xb^s_t$ ($\xb^{*s}_t$) follow the steady-state distribution when using $\Kb_t$ ($\Kb^*_t$) as a fix controller from the outset, i.e., 
\begin{equation*}
\Xb^s_t:=\mathbb{E}\left[\xb_t^s\xb_t^{s\top}\right] ~~~~~~ \Xb^{*s}_t:=\mathbb{E}\left[\xb_t^{*s}\xb_t^{*s\top}\right].
\end{equation*}
The regret is then decomposed into three terms, 
{\scriptsize
\begin{equation*}
    \begin{split}
        &\mathbb{E}\left[\sum_{t=1}^T (\xb_t^{\top}\Qb_t\xb_t + \ub_t^{\top}\Rb_t\ub_t) - ({\xb}_t^{*\top}\Qb_t\xb^*_t + \ub_t^{*\top}\Rb_t\ub^*_t)\right]\\
        =&\sum_{t=1}^T\left[\text{Tr}\big((\Qb_t+\Kb^{\top}_t\Rb_t\Kb_t)\Xb_t\big)-\text{Tr}\big((\Qb_t+\Kb^{\top}_t\Rb_t\Kb_t)\Xb^s_t\big)\right]\\
        +&\sum_{t=1}^T\left[\text{Tr}\big((\Qb_t+\Kb^{\top}_t\Rb_t\Kb_t)\Xb^s_t\big)-\text{Tr}\big((\Qb_t+\Kb^{*\top}_t\Rb_t\Kb^*_t)\Xb^{*s}_t\big)\right]\\
        + &\sum_{t=1}^T\left[\text{Tr}\big((\Qb_t+\Kb^{*\top}_t\Rb_t\Kb^*_t)\Xb^{*s}_t\big)-\text{Tr}\big((\Qb_t+\Kb^{*\top}_t\Rb_t\Kb^*_t)\Xb^*_t\big)\right],
    \end{split}
\end{equation*}
}
which we denote as Term I, Term II, and Term III, respectively. We will provide the upper bound for each one in the sequel.

{\bf Term II:} Based on Lemma \ref{L: Local convergence}, we have 
\begin{equation}\label{Regret bound Eq1}
\begin{split}
    &\text{dist}(\Kb_{t+1}, \Kb^*_{t})^2 \leq \alpha^2\text{dist}(\Kb_{t}, \Kb^*_{t})^2\\
    \leq &\alpha^2\left[2\text{dist}(\Kb_{t}, \Kb^*_{t-1})^2 + 2\text{dist}(\Kb^*_{t-1}, \Kb^*_{t})^2\right],
\end{split}
\end{equation}
where $\text{dist}(\cdot,\cdot)$ denotes the Riemannian distance based on an appropriate metric. Summing Equation \eqref{Regret bound Eq1} over $t$, we get
\begin{equation}\label{Regret bound Eq2}
\begin{split}
    &\sum_{t=2}^{T-1} \text{dist}(\Kb_{t+1}, \Kb^*_{t})^2 \\
    \leq &2\alpha^2\sum_{t=2}^{T-1}\text{dist}(\Kb_{t}, \Kb^*_{t-1})^2 + 2\alpha^2\sum_{t=2}^{T-1}\text{dist}(\Kb^*_{t-1}, \Kb^*_{t})^2.
\end{split}          
\end{equation}
Adding and subtracting $2\alpha^2\text{dist}(\Kb_{T}, \Kb^*_{T-1})^2$ to the right hand side, we get 
\begin{equation}\label{Regret bound Eq3}
\begin{split}
    &\sum_{t=1}^{T-1} \text{dist}(\Kb_{t+1}, \Kb^*_{t})^2 \\
    \leq &\text{dist}(\Kb_{2}, \Kb^*_{1})^2 - 2\alpha^2\text{dist}(\Kb_{T}, \Kb^*_{T-1})^2 \\
    + &2\alpha^2\sum_{t=1}^{T-1}\text{dist}(\Kb_{t+1}, \Kb^*_{t})^2 + 2\alpha^2\sum_{t=2}^{T}\text{dist}(\Kb^*_{t-1}, \Kb^*_{t})^2.
\end{split}    
\end{equation}
By choosing $\alpha$ such that $2\alpha^2 < 1$, Equation \eqref{Regret bound Eq3} can be re-arranged as follows
\begin{equation}\label{Regret bound Eq4}
\begin{split}
    &\sum_{t=1}^{T-1} \text{dist}(\Kb_{t+1}, \Kb^*_{t})^2\\ 
    \leq &\frac{\text{dist}(\Kb_{2}, \Kb^*_{1})^2 - 2\alpha^2\text{dist}(\Kb_{T}, \Kb^*_{T-1})^2}{1-2\alpha^2}\\
    + &\frac{2\alpha^2}{1-2\alpha^2}\sum_{t=2}^{T}\text{dist}(\Kb^*_{t}, \Kb^*_{t-1})^2.
\end{split}
\end{equation}
Based on Equation \eqref{Eq: Cost reformulation2}, we can see that 
\begin{align*}   h_t(\Kb_t)=&\text{Tr}\big((\Qb_t+\Kb^{\top}_t\Rb_t\Kb_t)\Xb^s_t\big)\\h_t(\Kb^*_t)=&\text{Tr}\big((\Qb_t+\Kb^{*\top}_t\Rb_t\Kb^*_t)\Xb^{*s}_t\big). 
\end{align*}
Then, Term II can be expressed as $\sum_{t=1}^T h_t({\Kb}_{t}) - h_t(\Kb^*_{t})$. Since $\Kb^*_t$ is a local minimizer of $h_t$, based on Assumption \ref{A: Hessian boundedness} (upper bound for Riemannian Hessian), we get
\begin{equation}\label{Regret bound Eq5}
    h_t({\Kb}_{t}) - h_t(\Kb^*_{t}) \leq \frac{L_g}{2}\text{dist}({\Kb}_t,\Kb^*_t)^2.
\end{equation}
Summing above over $t$ and applying the triangle inequality, we have
\begin{equation}\label{Regret bound Eq6}
\begin{split}
    &\sum_{t=1}^T h_t({\Kb}_{t}) - h_t(\Kb^*_{t})\\
    \leq &\frac{L_g}{2}\sum_{t=2}^T \left[2\text{dist}({\Kb}_t,\Kb^*_{t-1})^2 + 2\text{dist}(\Kb^*_{t-1}, \Kb^*_t)^2\right]\\
    + &\frac{L_g}{2}\text{dist}({\Kb}_1,\Kb^*_{1})^2\\
    \leq &\frac{L_g}{2}\text{dist}({\Kb}_1,\Kb^*_{1})^2 + L_g\sum_{t=2}^T  \text{dist}(\Kb^*_{t-1}, \Kb^*_t)^2\\
    + &L_g\left[\frac{\text{dist}(\Kb_{2}, \Kb^*_{1})^2 - 2\alpha^2\text{dist}(\Kb_{T}, \Kb^*_{T-1})^2}{1-2\alpha^2}\right]\\
    + &\frac{2L_g\alpha^2}{1-2\alpha^2}\sum_{t=2}^{T}\text{dist}(\Kb^*_{t}, \Kb^*_{t-1})^2,
\end{split}
\end{equation}
where the last inequality is based on Equation \eqref{Regret bound Eq4}.

{\bf Term I and Term III:} Suppose that the controller sequence $\{\Kb_t\}$ satisfies the condition $\Kb_t\in \mathcal{U}_t$ for all $t$, and let $\kappa = \sqrt{\nu}/\sigma$. Assume that $\norm{\Xb^s_{t+1} - \Xb^s_{t}}\leq \zeta$ for all $t$ for some $\zeta\leq \sigma^2/\kappa^2$. Then, by following the derivations of Lemmas 4.3 and 4.4 in \cite{cohen2018online}, it can be shown that the controller sequence $\{\Kb_t\}$ is $(\kappa, \frac{1}{2\kappa^2})$-sequentially strongly stable. We can then use Lemma \ref{L: Sequential strong stability (modified)} to get
\begin{equation}\label{Regret bound (revised) Eq3}
\begin{split}
    \norm{\Xb_t - \Xb^s_t}
    \leq &\exp^{\frac{-(t-1)}{2\kappa^2}}\kappa^2\norm{\Xb_1 - \Xb^s_1}\\
    + &\kappa^2\sum_{i=0}^{t-2}(1-\frac{1}{4\kappa^2})^{2i}\norm{\Xb^s_{t-i} - \Xb^s_{t-(i+1)}}.
\end{split}  
\end{equation}
Since $\Kb_t\in \mathcal{U}_t$ implies that for all $t$, the Riemannian distance between $\Kb_t$ and $\Kb^*_t$ is bounded, we can assume there exists a positive constant $L_s$ such that $\norm{\Xb^s_{t+1} - \Xb^s_{t}}\leq L_s\text{dist}(\Kb_{t+1}, \Kb_{t})$.
Therefore, to ensure \eqref{Regret bound (revised) Eq3}, it is just sufficient to have $\max_{t} \left[\text{dist}(\Kb_{t+1}, \Kb_{t})\right] \leq \sigma^2/L_s\kappa^2$, which is valid due to Lemmas \ref{L: Local convergence} and \ref{L: Bounded starting point} and the initialization condition. Based on Equation \eqref{Regret bound (revised) Eq3} and the facts that $\text{Tr}(\Qb_t), \text{Tr}(
\Rb_t)\leq C$ and $\norm{\Kb_t}\leq \kappa$, we have that
{\small
\begin{equation}\label{Regret bound (revised) Eq4}
\begin{split}
     \text{Term I}\leq &\sum_{t=1}^T \text{Tr}(\Qb_t+\Kb^{\top}_t\Rb_t\Kb_t)\norm{\Xb_t - \Xb^s_t}\\
    \leq &C(1+\kappa^2)\sum_{t=1}^T \kappa^2 \exp^{\frac{-(t-1)}{2\kappa^2}}\norm{\Xb_1 - \Xb^s_1}\\
    + &C(1+\kappa^2)\kappa^2\sum_{t=2}^T\sum_{j=2}^t(1-\frac{1}{4\kappa^2})^{2(t-j)}\norm{\Xb^s_{j} - \Xb^s_{j-1}}\\
    \leq &C(1+\kappa^2)\sum_{t=1}^T \kappa^2 \exp^{\frac{-(t-1)}{2\kappa^2}}\norm{\Xb_1 - \Xb^s_1}\\
    + &C(1+\kappa^2)\kappa^2\sum_{j=2}^T\norm{\Xb^s_{j} - \Xb^s_{j-1}}\sum_{t=j}^T(1-\frac{1}{4\kappa^2})^{2(t-j)}\\
    \leq &C(1+\kappa^2)\sum_{t=1}^T \kappa^2 \exp^{\frac{-(t-1)}{2\kappa^2}}\norm{\Xb_1 - \Xb^s_1}\\
    + &4C(1+\kappa^2)\kappa^4\sum_{j=2}^T\norm{\Xb^s_{j} - \Xb^s_{j-1}}.
\end{split}
\end{equation}
}

Similar to Equation \eqref{Regret bound (revised) Eq3}, if $\forall t,\;\norm{\Xb^{*s}_{t+1} - \Xb^{*s}_{t}} \leq \zeta$ for some $\zeta \leq \sigma^2/\kappa^2$, we also have
\begin{equation}\label{Regret bound (revised) Eq5}
    \begin{split}
    \norm{\Xb^*_t - \Xb^{*s}_t}
    \leq &\exp^{\frac{-(t-1)}{2\kappa^2}}\kappa^2\norm{\Xb^*_1 - \Xb^{*s}_1}\\
    + &\kappa^2\sum_{i=0}^{t-2}(1-\frac{1}{4\kappa^2})^{2i}\norm{\Xb^{*s}_{t-i} - \Xb^{*s}_{t-(i+1)}}.
\end{split}
\end{equation}
By the smoothness of the steady-state covariance matrix, we have $\norm{\Xb^{*s}_{t+1} - \Xb^{*s}_{t}}\leq L_s \text{dist}(\Kb^*_{t+1}, \Kb^*_{t})$, which together with the assumption that  
$$\text{dist}(\Kb^*_{t+1}, \Kb^*_{t})\leq \frac{\sigma^4}{L_s\nu}=\frac{\sigma^2}{L_s\kappa^2},$$
guarantees $\norm{\Xb^{*s}_{t+1} - \Xb^{*s}_{t}}\leq \sigma^2/\kappa^2$. Therefore, Equation \eqref{Regret bound (revised) Eq5} holds, and we have that 
\begin{equation}\label{Regret bound Eq11}
\begin{split}
     \text{Term III}
    \leq &C(1+\kappa^2)\sum_{t=1}^T \kappa^2 \exp^{\frac{-(t-1)}{2\kappa^2}}\norm{\Xb^*_1 - \Xb^{*s}_1}\\
    + &4C(1+\kappa^2)\kappa^4\sum_{j=2}^T\norm{\Xb^{*s}_{j} - \Xb^{*s}_{j-1}}.
\end{split}
\end{equation}
Based on Equations \eqref{Regret bound Eq6}, \eqref{Regret bound (revised) Eq4} and \eqref{Regret bound Eq11}, we conclude that the regret bound is 
$$O\left(\sum_{t=1}^{T-1}\left[\text{dist}(\Kb_t,\Kb_{t+1}) + \text{dist}(\Kb^*_t,\Kb^*_{t+1})\right]\right).$$
Next, we further show that $\sum_{t=1}^{T-1}\text{dist}(\Kb_t,\Kb_{t+1}) = O\left(\sum_{t=1}^{T-1}\text{dist}(\Kb^*_t,\Kb^*_{t+1})\right)$. Applying the triangle inequality and using Lemma \ref{L: Local convergence}, we have
\begin{equation}\label{Regret bound Eq12}
\begin{split}
    \sum_{t=1}^{T-1}\text{dist}(\Kb_t,\Kb_{t+1})
    \leq &\sum_{t=1}^{T-1}\text{dist}(\Kb_t,\Kb^*_t) + \text{dist}(\Kb^*_t,\Kb_{t+1})\\
    \leq &\sum_{t=1}^{T-1}(1+\alpha)\text{dist}(\Kb_t,\Kb^*_t).
\end{split}
\end{equation}
Again, based on Lemma \ref{L: Local convergence}, we have $\forall t$,
\begin{equation}\label{Corollary revised Eq1}
\begin{split}
    \text{dist}(\Kb_{t+1}, \Kb^*_{t+1}) &\leq \text{dist}(\Kb_{t+1}, \Kb^*_{t}) + \text{dist}(\Kb^*_{t}, \Kb^*_{t+1})\\
    &\leq \alpha\text{dist}(\Kb_{t}, \Kb^*_{t}) + \text{dist}(\Kb^*_{t}, \Kb^*_{t+1}).
\end{split}
\end{equation}
Then, by expanding the recursion above, we get
\begin{equation}\label{Corollary revised Eq2}
\begin{split}
    &\text{dist}(\Kb_{t+1}, \Kb^*_{t+1})\\
    \leq &\alpha^t\text{dist}(\Kb_{1}, \Kb^*_{1}) + \sum_{i=0}^{t-1} \alpha^i\text{dist}(\Kb^*_{t-i}, \Kb^*_{(t+1)-i}).
\end{split}
\end{equation}
Summing above over $t$, we obtain
\begin{equation}\label{Corollary revised Eq3}
\begin{split}
    &\sum_{t=0}^{T-1}\left[\alpha^t\text{dist}(\Kb_{1}, \Kb^*_{1}) + \sum_{i=0}^{t-1} \alpha^i\text{dist}(\Kb^*_{t-i}, \Kb^*_{(t+1)-i})\right]\\
    \leq &\frac{1}{1-\alpha} \text{dist}(\Kb_{1}, \Kb^*_{1}) + \sum_{t=1}^{T-1}\sum_{j=1}^{t} \alpha^{t-j}\text{dist}(\Kb^*_{j}, \Kb^*_{j+1})\\
    = &\frac{1}{1-\alpha} \text{dist}(\Kb_{1}, \Kb^*_{1}) + \sum_{j=1}^{T-1}\text{dist}(\Kb^*_{j}, \Kb^*_{j+1})\sum_{t=j}^{T-1} \alpha^{t-j}\\
    \leq &\frac{1}{1-\alpha} \text{dist}(\Kb_{1}, \Kb^*_{1}) + \frac{1}{1-\alpha}\sum_{j=1}^{T-1}\text{dist}(\Kb^*_{j}, \Kb^*_{j+1}),
\end{split}    
\end{equation}
Applying Equation \eqref{Corollary revised Eq3} to Equation \eqref{Regret bound Eq12}, we conclude that the regret bound is
$$O\left(\sum_{t=1}^{T-1}\text{dist}(\Kb^*_t,\Kb^*_{t+1})\right).$$

\section{Supplementary Lemmas}\label{Sec: B}    
\begin{lemma}\label{L: Sequential strong stability (modified)}
    Consider a sequence of linear controllers $\{\Kb_t\}$ that is $(\kappa, \gamma)$-sequentially strongly stable with respect to an LTI system $(\Ab,\Bb)$. Denote by $\Xb^s_t$ the steady-state covariance matrix corresponding to $\Kb_t$ and by $\Xb_t$ the state covariance matrix at iteration $t$ when the policy $\ub_t = \Kb_t\xb_t$ is applied. Then, 
    \begin{equation*}
        \begin{split}
            \norm{\Xb_{t} - \Xb^s_{t}}
        \leq &\exp^{-(t-1)\gamma}\kappa^2\norm{\Xb_1 - \Xb^s_1}\\
        + &\kappa^2\sum_{i=0}^{t-2}(1-\frac{\gamma}{2})^{2i}\norm{\Xb^s_{t-i} - \Xb^s_{t-1-i}}.
        \end{split}
    \end{equation*}
\end{lemma}
\begin{proof}
    Based on the definition, we have the following equations
    \begin{equation*}
    \begin{split}
        \Xb_{t+1} &= (\Ab + \Bb\Kb_t)\Xb_t(\Ab + \Bb\Kb_t)^{\top} + \Wb,\\
        \Xb^s_{t} &= (\Ab + \Bb\Kb_t)\Xb^s_t(\Ab + \Bb\Kb_t)^{\top} + \Wb.
    \end{split}
    \end{equation*}
    Calculating the difference between the above equations and using the fact that $(\Ab+\Bb\Kb_t)=\Hb_t\Lb_t\Hb_t^{-1}$, we have
    \begin{equation}\label{Sequential strong stability (modified) Eq1}
    \begin{split}
        &\Xb_{t+1} - \Xb^s_{t}
        = \Hb_t\Lb_t\Hb_t^{-1}(\Xb_{t} - \Xb^s_{t})(\Hb_t^{-1})^{\top}\Lb_t^{\top}\Hb_t^{\top}.
    \end{split}    
    \end{equation}
    Subtracting $\Xb^s_{t+1}$ from both sides of the above, we get
    \begin{equation}\label{Sequential strong stability (modified) Eq2}
     \begin{split}
        &\Hb_{t+1}^{-1}(\Xb_{t+1} - \Xb^s_{t+1})(\Hb_{t+1}^{-1})^{\top}\\
        = &\Hb_{t+1}^{-1}\Hb_t\Lb_t\Hb_t^{-1}(\Xb_{t} - \Xb^s_{t})(\Hb_t^{-1})^{\top}\Lb_t^{\top}\Hb_t^{\top}(\Hb_{t+1}^{-1})^{\top}\\
        + &\Hb_{t+1}^{-1}(\Xb^s_{t} - \Xb^s_{t+1})(\Hb_{t+1}^{-1})^{\top}.
    \end{split}    
    \end{equation}
    Based on Equation \eqref{Sequential strong stability (modified) Eq2} and following the definition of sequential strong stability, we derive
    \begin{equation}\label{Sequential strong stability (modified) Eq3}
    \begin{split}
        &\norm{\Hb_{t+1}^{-1}(\Xb_{t+1} - \Xb^s_{t+1})(\Hb_{t+1}^{-1})^{\top}}\\
        \leq &\norm{\Hb_{t+1}^{-1}\Hb_t\Lb_t\Hb_t^{-1}(\Xb_{t} - \Xb^s_{t})(\Hb_t^{-1})^{\top}\Lb_t^{\top}\Hb_t^{\top}(\Hb_{t+1}^{-1})^{\top}}\\
        + &\norm{\Hb_{t+1}^{-1}}^2\norm{\Xb^s_{t} - \Xb^s_{t+1}}\\
        \leq & (1-\gamma)^2(1+\frac{\gamma}{2})^2\norm{\Hb_t^{-1}(\Xb_{t} - \Xb^s_{t})(\Hb_t^{-1})^{\top}}\\
        +&\norm{\Hb_{t+1}^{-1}}^2\norm{\Xb^s_{t} - \Xb^s_{t+1}}\\
        \leq &(1-\frac{\gamma}{2})^2\norm{\Hb_t^{-1}(\Xb_{t} - \Xb^s_{t})(\Hb_t^{-1})^{\top}}\\
        +&\norm{\Hb_{t+1}^{-1}}^2\norm{\Xb^s_{t} - \Xb^s_{t+1}}.
    \end{split}
    \end{equation}
    Unfolding Equation \eqref{Sequential strong stability (modified) Eq3}, we have
    \begin{equation}\label{Sequential strong stability (modified) Eq4}
    \begin{split}
        &\norm{\Hb_{t+1}^{-1}(\Xb_{t+1} - \Xb^s_{t+1})(\Hb_{t+1}^{-1})^{\top}}\\
        \leq &(1-\frac{\gamma}{2})^{2t}\norm{\Hb_1^{-1}(\Xb_1 - \Xb^s_1)(\Hb_1^{-1})^{\top}}\\
        + &\sum_{i=0}^{t-1}\norm{\Hb_{t+1-i}^{-1}}^2(1-\frac{\gamma}{2})^{2i}\norm{\Xb^s_{t+1-i} - \Xb^s_{t-i}}\\
        \leq &\exp^{-\gamma t}\norm{\Hb_1^{-1}(\Xb_1 - \Xb^s_1)(\Hb_1^{-1})^{\top}}\\
        + &\sum_{i=0}^{t-1}\norm{\Hb_{t+1-i}^{-1}}^2(1-\frac{\gamma}{2})^{2i}\norm{\Xb^s_{t+1-i} - \Xb^s_{t-i}}.
        \end{split}
    \end{equation}
    Again, based on the definition of sequential strong stability $\forall t,\;\norm{\Hb_{t}^{-1}} \leq \frac{1}{\alpha^{\prime}},\; \norm{\Hb_t}\leq \beta^{\prime} \text{ and } \frac{\beta^{\prime}}{\alpha^{\prime}}=\kappa$, we have  
    \begin{equation}\label{Sequential strong stability (modified) Eq5}
    \begin{split}
        \norm{\Xb_{t+1} - \Xb^s_{t+1}}
        \leq &\exp^{-\gamma t}\kappa^2\norm{\Xb_1 - \Xb^s_1}\\
        + &\kappa^2\sum_{i=0}^{t-1}(1-\frac{\gamma}{2})^{2i}\norm{\Xb^s_{t+1-i} - \Xb^s_{t-i}}.
    \end{split}  
    \end{equation}
\end{proof}

\begin{lemma}\label{L: Local convergence}
    Suppose that Assumptions \ref{A: Locally optimal stationary} to \ref{A: Bounded trace of a stable controller} hold. Then, if there exists an $\alpha\in(0,1/\sqrt{2})$ for which $\forall t,\;\text{dist}(\Kb_{t},\Kb^*_t) \leq \min \left\{\frac{c}{(L_g/\mu_g)-\alpha/2},\frac{\alpha}{(L_HL_g^2/\mu_g^3)}\right\}$, by selecting $\eta_{t} = \min\{1, s_{\Kb_{t}}\}$, the following inequality holds:
    \begin{equation*}
        \text{dist}(\Kb_{t+1}, \Kb^*_{t}) \leq \alpha\text{dist}(\Kb_{t}, \Kb^*_{t}),
    \end{equation*}
    for $t=1,\ldots,T$. 
\end{lemma}
\begin{proof}
    Based on Assumption \ref{A: Hessian boundedness}, we have that
$$
\mu_g\norm{\Gb_{t}}^2_{g_{\Kb_{t}}} \leq \langle \text{Hess }{h_t}_{\Kb_{t}}[\Gb_{t}], \Gb_{t}\rangle,
$$
so using Cauchy-Schwartz inequality at $\Kb_{t}$, for the Newton direction $\Gb_{t}$, we have that
        \begin{equation}\label{Local convergence Eq1}
    \norm{\Gb_{t}}_{g_{\Kb_{t}}}\leq \frac{1}{\mu_g}\norm{\text{grad }{h_t}_{\Kb_{t}}}_{g_{\Kb_{t}}}.
        \end{equation}
    Define a curve $r: [0,s_{\Kb_{t}}]\to \Tilde{\Sc}$ such that $r(\eta) = \Kb_{t} + \eta \Gb_{t}$ and consider a smooth vector field $E(\eta)$ which is parallel along the curve $r$. Define another scalar function $\phi:[0,s_{\Kb_{t}}]\to \mathrm{R}$ such that $\phi(\eta) = \langle\text{grad }{h_t}_{r(\eta)}, E(\eta) \rangle$. 
    We then have
    \begin{equation}\label{Local convergence Eq2}
    \begin{split}
        &\phi^{\prime}(\eta)
        =\langle\text{Hess }{h_t}_{r(\eta)}[\Gb_{t}], E(\eta)\rangle.
    \end{split}
    \end{equation}
    Since
    \begin{equation}\label{Local convergence Eq3}
        \phi(\eta) = \phi(0) + \eta\phi^{\prime}(0) + \int_{0}^{\eta}\left(\phi^{\prime}(\tau)-\phi^{\prime}(0)\right) d\tau,
    \end{equation}
    by substituting \eqref{Local convergence Eq2} into \eqref{Local convergence Eq3}, we derive
    \begingroup
\small
        \begin{equation}\label{Local convergence Eq4}
        \begin{split}
            \vphantom{\int_0^1}&\phi(\eta_{t})
            = \langle \text{grad }{h_t}_{\Kb_{t+1}}, E(\eta_{t}) \rangle\\ 
            = &(1-\eta_{t})\langle \text{grad }{h_t}_{\Kb_{t}}, E(0) \rangle\\
            + &\int_{0}^{\eta_{t}} \langle \text{Hess }{h_t}_{r(\tau)}[\Gb_{t}], E(\tau) \rangle  - \langle \text{Hess }{h_t}_{r(0)}[\Gb_{t}], E(0)\rangle d\tau\\
            =&(1-\eta_{t})\langle \Pc^r_{0,\eta_{t}}\text{grad }{h_t}_{\Kb_{t}}, E(\eta_{t}) \rangle\\
            + &\int_{0}^{\eta_{t}} \langle (\text{Hess }{h_t}_{r(\tau)}-\Pc^r_{0,\tau}\text{Hess }{h_t}_{r(0)})[\Gb_{t}], E(\tau) \rangle d\tau,
        \end{split}
        \end{equation}
        \endgroup
    where $\Pc^r_{0,\tau}$ denotes the parallel transport operator from $0$ to $\tau$ along the curve $r$, and the last equality is due to the linear isometry property of parallel transport. Noting that $\forall t$, the Hessian operator $\text{Hess }{h_t}_{r(\eta)}$ is smooth in $\eta$, there exists a general constant $L_H$ such that
    \begingroup
\small
    \begin{equation}\label{Local convergence Eq5}
    \begin{split}
        &\norm{(\text{Hess }{h_t}_{r(\tau)}-\Pc^r_{0,\tau}\text{Hess }{h_t}_{r(0)})[\Gb_{t}]}_{g_{r(\tau)}}
        \leq L_H\tau\norm{\Gb_{t}}^2_{g_{\Kb_{t}}}.
    \end{split}   
    \end{equation}
    \endgroup
    Also since the parallel transport operator conserves the inner product, we have
    \begin{equation}\label{Local convergence Eq6}
        \norm{\Pc^r_{0,\eta_{t}}\text{grad }{h_t}_{\Kb_{t}}}_{g_{\Kb_{t+1}}} = \norm{\text{grad }{h_t}_{\Kb_{t}}}_{g_{\Kb_{t}}}.
    \end{equation}
    Then, by choosing the parallel vector field $E(\eta)$ satisfying $E(\eta_{t}) = \text{grad }{h_t}_{\Kb_{t+1}}$, based on Equations \eqref{Local convergence Eq5} and \eqref{Local convergence Eq6} we get

    \begin{equation}\label{Local convergence Eq7}
    \begin{split}
        \vphantom{\int_{0}^t}\phi(\eta_{t}) &= \langle \text{grad }{h_t}_{\Kb_{t+1}}, \text{grad }{h_t}_{\Kb_{t+1}} \rangle\\
        &\leq (1-\eta_{t})\norm{\text{grad }{h_t}_{\Kb_{t}}}_{g_{\Kb_{t}}}\norm{\text{grad }{h_t}_{\Kb_{t+1}}}_{g_{\Kb_{t+1}}}\\
        &+\int_{0}^{\eta_{t}}\left[ L_H\tau\norm{\Gb_{t}}^2_{g_{\Kb_{t}}}\norm{E(\tau)}_{g_{r(\tau)}} \right]d\tau\\
        &= (1-\eta_{t})\norm{\text{grad }{h_t}_{\Kb_{t}}}_{g_{\Kb_{t}}}\norm{\text{grad }{h_t}_{\Kb_{t+1}}}_{g_{\Kb_{t+1}}}\\
        &+\int_{0}^{\eta_{t}}\left[ L_H\tau\norm{\Gb_{t}}^2_{g_{\Kb_{t}}}\norm{\text{grad }{h_t}_{\Kb_{t+1}}}_{g_{\Kb_{t+1}}} \right]d\tau\\
        &=(1-\eta_{t})\norm{\text{grad }{h_t}_{\Kb_{t}}}_{g_{\Kb_{t}}}\norm{\text{grad }{h_t}_{\Kb_{t+1}}}_{g_{\Kb_{t+1}}}\\
        &+ \frac{\eta_{t}^2}{2}L_H\norm{\Gb_{t}}^2_{g_{\Kb_{t}}}\norm{\text{grad }{h_t}_{\Kb_{t+1}}}_{g_{\Kb_{t+1}}},
    \end{split}            
    \end{equation}
    where the first inequality is based on Cauchy–Schwarz inequality, and the following equality is due to the fact that the length of the parallel vector field is constant. From above, we conclude that
    \begin{equation}\label{Local convergence Eq8}
    \begin{split}
        &\norm{\text{grad }{h_t}_{\Kb_{t+1}}}_{g_{\Kb_{t+1}}}\\
        \leq &(1-\eta_{t})\norm{\text{grad }{h_t}_{\Kb_{t}}}_{g_{\Kb_{t}}} + \frac{\eta_{t}^2}{2}L_H\norm{\Gb_{t}}^2_{g_{\Kb_{t}}}\\
        \leq &(1-\eta_{t})\norm{\text{grad }{h_t}_{\Kb_{t}}}_{g_{\Kb_{t}}} + \frac{\eta^2_{t} L_H}{2\mu_g^2}\norm{\text{grad }{h_t}_{\Kb_{t}}}^2_{g_{\Kb_{t}}},
    \end{split}            
    \end{equation}
    where the last inequality is based on \eqref{Local convergence Eq1}. Next, select a tangent vector $F_{t+1}\in T_{\Kb_{t+1}}\Tilde{\Sc}$ such that the curve $\xi(\eta):=\exp_{\Kb_{t+1}}[\eta F_{t+1}]$ is the geodesic between $\xi(0)=\Kb_{t+1}$ and $\xi(1)=\Kb^*_t$, and also $\xi^{\prime}(0) = F_{t+1}$. 
    Then, for a parallel vector field $E(\eta)$ along $\xi$, define a scalar function $\psi:[0,1]\mapsto \mathrm{R}$ such that $\psi(\eta) = \langle 
\text{grad }{h_t}_{\xi(\eta)} ,E(\eta) \rangle$. Similar to \eqref{Local convergence Eq2}, we have that
    \begin{equation}\label{Local convergence Eq9}
        \psi^{\prime}(\eta) = \langle \text{Hess }{h_t}_{\xi(\eta)}[\xi^{\prime}(\eta)], E(\eta)\rangle.
    \end{equation}
    As the velocity of a geodesic curve is parallel, by choosing $E(\eta) = \xi^{\prime}(\eta)$ and based on \eqref{Local convergence Eq9}, we get
        \begin{equation}\label{Local convergence Eq10}
        \begin{split}
            \psi(1)
            = &\psi(0) + \int_{\tau=0}^1 \psi^{\prime}(\tau)d\tau\\
            = &\langle\text{grad }{h_t}_{\Kb_{t+1}}, F_{t+1}\rangle \\ 
            +&\int_{\tau=0}^1 \langle \text{Hess }{h_t}_{\xi(\tau)}[\xi^{\prime}(\tau)], \xi^{\prime}(\tau)\rangle d\tau.
        \end{split}    
        \end{equation}    
    Since $\psi(1) = 0$ (i.e., $\Kb^*_t$ is a local minimum), based on the Hessian boundedness assumption and the fact that $\norm{\xi^{\prime}(\tau)}_{g_{\xi(\tau)}} = \norm{F_{t+1}}_{g_{\Kb_{t+1}}}$ for $\tau\in[0,1]$, we have 
    \begin{equation}\label{Local convergence Eq11}
    \begin{split}
        \mu_g\norm{F_{t+1}}^2_{g_{\Kb_{t+1}}} \leq  &\int_{\tau=0}^1 \langle \text{Hess }{h_t}_{\xi(\tau)}[\xi^{\prime}(\tau)], \xi^{\prime}(\tau)\rangle d\tau\\
        \vphantom{\int_{\tau=0}^1}= &-\langle\text{grad }{h_t}_{\Kb_{t+1}}, F_{t+1}\rangle\\
        \leq &\norm{\text{grad }{h_t}_{\Kb_{t+1}}}_{g_{\Kb_{t+1}}} \vphantom{\int_{\tau=0}^1}\norm{F_{t+1}}_{g_{\Kb_{t+1}}}\\
        \vphantom{\int_{\tau=0}^1} \Rightarrow \mu_g\norm{F_{t+1}}_{g_{\Kb_{t+1}}}\leq &\norm{\text{grad }{h_t}_{\Kb_{t+1}}}_{g_{\Kb_{t+1}}}.
    \end{split}
    \end{equation}
    Note that $\norm{F_{t+1}}_{g_{\Kb_{t+1}}} = \text{dist}(\Kb_{t+1},\Kb^*_t)$, where $\text{dist}(\cdot,\cdot)$ denotes the Riemannian distance function. 
    Next, based on the smoothness of $\text{grad }h_t$ and the boundedness of $\cup \{\mathcal{U}_t\}$, we have 
    \begin{equation}\label{Local convergence Eq12}
        \norm{\text{grad }{h_t}_{\Kb_{t}}}_{g_{\Kb_{t}}} \leq L_g \text{dist}(\Kb_t,\Kb^*_t).
    \end{equation}
    Substituting Equations \eqref{Local convergence Eq11} and \eqref{Local convergence Eq12} into Equation \eqref{Local convergence Eq8}, we derive
    \begin{equation}\label{Local convergence Eq13}
    \begin{split}
        &\text{dist}(\Kb_{t+1},\Kb^*_t) \\
        \leq &(1-\eta_{t})\frac{L_g}{\mu_g}\text{dist}(\Kb_{t},\Kb^*_t) + \frac{\eta^2_{t} L_HL_g^2}{2\mu_g^3}\text{dist}(\Kb_{t},\Kb^*_t)^2.
    \end{split}
    \end{equation}
    Notice that the mapping $\Kb\mapsto \mathbb{L}((\Ab+\Bb\Kb)^\top, \mathcal{Q}_{\Kb})$ is smooth (as the mapping $\mathcal{Q}$ is selected to be smooth),  so based on the continuity of the maximum
    eigenvalue, there exists a positive constant $c$ such that
    \begin{equation}\label{Local convergence Eq14}
        s_{\Kb_{t}} \geq \frac{c}{\norm{\Gb_{t}}_{g_{\Kb_{t}}}} \geq \frac{c\mu_g}{L_g\text{dist}(\Kb_{t},\Kb^*_t)},
    \end{equation}
    where the second inequality is based on Equations \eqref{Local convergence Eq1} and \eqref{Local convergence Eq12}. For an $\alpha\in(0,1/\sqrt{2})$, based on the assumptions of the lemma, we have $\forall t$
        \begin{equation}\label{Local convergence (modified) Eq1}
            \text{dist}(\Kb_{t},\Kb^*_t) \leq \min \left\{\frac{c}{(L_g/\mu_g)-\alpha/2},\frac{\alpha}{(L_HL_g^2/\mu_g^3)}\right\}.
        \end{equation}
    Then, by selecting $\eta_{t}=\min\{1, s_{\Kb_{t}}\}$, we can guarantee
    $$
    \text{dist}(\Kb_{t+1},\Kb^*_t) \leq \alpha \text{dist}(\Kb_t,\Kb^*_t),
    $$
    because if $s_{\Kb_t} \geq 1$, we have $\eta_t=1$ in \eqref{Local convergence Eq13}, and the result is immediate by observing \eqref{Local convergence (modified) Eq1}. Otherwise, if $s_{\Kb_t} < 1$, we have $\eta_t=s_{\Kb_t}$ and
    \begin{equation}\label{Local convergence (modified) Eq2}
    \begin{split}
        &(1-\eta_{t})\frac{L_g}{\mu_g} + \frac{\eta^2_{t}L_HL_g^2}{2\mu_g^3}\text{dist}(\Kb_{t},\Kb^*_t)\\
        \leq & (1-s_{\Kb_{t}})\frac{L_g}{\mu_g} + \frac{L_HL_g^2}{2\mu_g^3}\text{dist}(\Kb_{t},\Kb^*_t)\\
        \leq & \frac{L_g}{\mu_g} - \left(\frac{L_g}{\mu_g}-\frac{\alpha}{2}\right) + \frac{\alpha}{2} = \alpha,        
    \end{split}
    \end{equation}
    where the second inequality is based on Equations \eqref{Local convergence (modified) Eq1} and \eqref{Local convergence Eq14}. Therefore, the proof is complete.

    
\end{proof}

\begin{lemma}\label{L: Bounded starting point}
    Suppose that for some $\vartheta > 0$, $\text{dist}(\Kb_{1}, \Kb^*_{1})\leq \vartheta$, $\text{dist}(\Kb^*_{t}, \Kb^*_{t+1})\leq (1-\alpha)\vartheta,\;\forall t$, and the assumptions of Lemma \ref{L: Local convergence} hold. Then, we have the following inequality
    \begin{equation*}
        \text{dist}(\Kb_{t}, \Kb^*_{t})\leq \vartheta,\;\forall t.
    \end{equation*}
\end{lemma}
\begin{proof}
    When $t=1$, we have $\text{dist}(\Kb_{1}, \Kb^*_{1}) \leq \vartheta$. Suppose that $\text{dist}(\Kb_{t}, \Kb^*_{t})\leq \vartheta$ holds; then for iteration $(t+1)$, we have
    \begin{equation}\label{Bounded starting point Eq1}
    \begin{split}
        \text{dist}(\Kb_{t+1}, \Kb^*_{t+1}) &\leq \text{dist}(\Kb_{t+1}, \Kb^*_{t}) + \text{dist}(\Kb^*_{t}, \Kb^*_{t+1})\\
        &\leq \alpha\text{dist}(\Kb_{t}, \Kb^*_{t}) + \text{dist}(\Kb^*_{t}, \Kb^*_{t+1})\\
        &\leq \alpha \vartheta + (1-\alpha)\vartheta\\
        &=\vartheta,
    \end{split}
    \end{equation}
    where the second inequality is based on Lemma \ref{L: Local convergence}. The result is proved by induction.
\end{proof}


\bibliographystyle{abbrvnat}        
\bibliography{references}           



\end{document}